\documentclass[letterpaper, 12pt]{article}

\usepackage[top=1.5in, bottom=1.5in, left=1.5in, right=1.5in]{geometry}

\usepackage{amsmath, amsfonts, amsthm}
\usepackage{graphicx}
\usepackage[all]{xy}
\usepackage{cite}
\usepackage{stackrel}
\usepackage{hyperref} 

\renewcommand{\S}{\Sigma}
\newcommand{\s}{\sigma}
\renewcommand{\b}{\beta}
\newcommand{\f}{\phi}
\newcommand{\G}{\Gamma}
\newcommand{\Gtilde}{\widetilde{\Gamma}}
\newcommand{\U}{\Upsilon}
\newcommand{\Co}{\mathbf{Comp}}
\newcommand{\F}{\mathbf{F}}
\newcommand{\Z}{\mathbb{Z}}
\newcommand{\N}{\mathbb{N}}
\newcommand{\E}{\mathbb{E}}
\newcommand{\Q}{\mathcal{Q}}

\theoremstyle{plain} \newtheorem{theorem}{Theorem}

\theoremstyle{plain} 
\theoremstyle{plain} \newtheorem{corollary}[theorem]{Corollary}
\theoremstyle{plain} \newtheorem{lemma}[theorem]{Lemma}
\theoremstyle{plain} 
\theoremstyle{plain} 

\theoremstyle{plain}  \newtheorem{definition}{Definition}
\theoremstyle{remark} 
\theoremstyle{remark} 

\begin{document}

\title{More on the phi = beta Conjecture and Eigenvalues of Random Graph Lifts}

\author{Edward Lui \\ Department of Computer Science \\ Cornell University \\ \texttt{luied@cs.cornell.edu} \and 
Doron Puder\\ Einstein Institute of Mathematics \\ Hebrew University, Jerusalem\\ \texttt{doronpuder@gmail.com}}

\maketitle

\begin{abstract}
Let $G$ be a connected graph, and let $\lambda_1$ and $\rho$ denote the spectral radius of $G$ and the universal cover of $G$, respectively. In \cite{Fri03}, Friedman has shown that almost every $n$-lift of $G$ has all of its new eigenvalues bounded by $O(\lambda_1^{1/2}\rho^{1/2})$. In \cite{LP10}, Linial and Puder have improved this bound to $O(\lambda_1^{1/3}\rho^{2/3})$. Friedman had conjectured that
this bound can actually be improved to $\rho + o_n(1)$ (e.g., see
\cite{Fri03,HLW06}).

In \cite{LP10}, Linial and Puder have formulated two new categorizations of formal words, namely $\phi$ and $\beta$, which assign a non-negative integer or infinity to each word. They have shown that for every word $w$, $\phi(w) = 0$ iff $\beta(w) = 0$, and $\phi(w) = 1$ iff $\beta(w) = 1$. They have conjectured that $\phi(w) = \beta(w)$ for every word $w$, and have run extensive numerical simulations that strongly suggest that this conjecture is true. This conjecture, if proven true, gives us a very promising approach to proving a slightly weaker version of Friedman's conjecture, namely the bound $O(\rho)$ on the new eigenvalues (see \cite{LP10}). 

In this paper, we make further progress towards proving this important conjecture by showing that $\phi(w) = 2$ iff $\beta(w) = 2$ for every word $w$. 
\end{abstract}

\section{Introduction}

Let $G$ be a connected finite graph with oriented edges. An
\emph{$n$-lift} of $G$ is any graph that has an $n$-fold covering
map onto $G$. Equivalently, an $n$-lift of $G$ is any graph $H$ with
vertex set $V(H) = V \times \{1,...,n\}$, whose edge set $E(H)$ is
obtained as follows: for every oriented edge $(u,v) \in E(G)$, we
choose any permutation $\sigma_{(u,v)} \in S_n$ and add an
(undirected) edge between $(u,i)$ and $(v,\sigma_{(u,v)}(i))$ for $i
= 1, \ldots, n$.

Now, let $E(G) = \{g_1, \ldots, g_k\}$. Every choice of $k$
permutations in $S_n$ determines an $n$-lift of $G$. The random
graph model we consider is the probability space $L_n(G)$ of
$n$-lifts of $G$, with sample space ${S_n}^k$ and uniform
distribution. We note that in the case where $G$ is a single vertex
with $d/2$ self-loops (with $d$ even), this random graph model is
the same as the ``permutation model'' for random $d$-regular graphs.
For background on lifts and random lifts, see \cite{LR05, AL06,
ALM02, HLW06}.

Our main interest is in the eigenvalues of (the adjacency matrix of)
random lifts of graphs. Let $H$ be any $n$-lift of $G$. The
projection $\pi : V(H) \to V(G)$ defined by $\pi(u,i) = u$ is the
natural covering map from $H$ onto $G$. It can be easily verified
that if $f$ is an eigenfunction of $G$, then $f \circ \pi$ is an
eigenfunction of $H$ with the same eigenvalue as $f$. The $|V(G)|$
eigenvalues of $H$ corresponding to the $|V(G)|$ such eigenfunctions
are said to be \emph{old}, while the remaining $n|V(G)| - |V(G)|$
eigenvalues of $H$ are said to be \emph{new}. (We consider the
multiset of  eigenvalues with multiplicity, so that a new eigenvalue
can have the same value as an old eigenvalue.)

Let $\lambda_1$ and $\rho$ denote the spectral radius of $G$ and the
universal cover of $G$, respectively. For $H \in L_n(G)$, let
$\mu_{max}(H) = \max\{|\mu| : \mu$ is a new eigenvalue of $H\}$. In
\cite{Fri03}, Friedman showed that almost every $n$-lift $H \in
L_n(G)$ satisfies $\mu_{max}(H) \leq \lambda_1^{1/2}\rho^{1/2} +
o_n(1)$. In \cite{LP10}, Linial and Puder improved this bound to
$\mu_{max}(H) \leq \max\left\{1,
3\left(\frac{\rho}{\lambda_1}\right)^{2/3}\right\} \cdot
\lambda_1^{1/3}\rho^{2/3} + o_n(1)$ for almost every $n$-lift $H \in
L_n(G)$.

For the special case where $G$ is a single vertex with $d/2$
self-loops (i.e., for the permutation model of random $d$-regular
graphs, with $d$ even), $\mu_{max}(H)$ corresponds to $\lambda(H) :=
\max\{|\lambda_2|, |\lambda_n|\}$, where $\lambda_1 \geq \lambda_2
\geq \cdots \geq \lambda_n$ are the eigenvalues of $H$. There is a
large body of literature concerning $\lambda(H)$. These studies are
motivated by the fact that $\lambda(H)$ controls the expansion
properties of $H$ and the rate of convergence of the random walk on
$H$ to the stationary distribution (see \cite{HLW06}).

For the permutation model, Friedman's result in \cite{Fri03} states
that almost every random $n$-vertex $d$-regular graph $H$ satisfies
$\lambda(H) \leq \sqrt{2d\sqrt{d-1}} + o_n(1)$, which is a slight
improvement of the result of Broder and Shamir in \cite{BS87}.
Linial and Puder's result in \cite{LP10} states that almost every
random $n$-vertex $d$-regular graph $H$ satisfies $\lambda(H) \leq
O(d^{2/3})$, and more specifically, $\lambda(H) \leq (4d(d-1))^{1/3}
+ o_n(1)$ for $d \geq 107$.

For various models of random $d$-regular graphs (including this
specific permutation model), Friedman had shown that almost every
random $n$-vertex $d$-regular graph $H$ with $d \geq 3$ satisfies
$\lambda(H) \leq 2\sqrt{d-1} + o_n(1)$ (see \cite{Fri08}). The
Alon-Boppana bound (see \cite{Nil91,Fri03}) shows that $\lambda(H)
\geq 2\sqrt{d-1} - o_n(1)$ for every $n$-vertex $d$-regular graph
$H$, so Friedman's result cannot be improved significantly, if at
all.

The results of \cite{BS87}, \cite{Fri03}, and \cite{LP10} all use
the \emph{Trace Method}, which involves estimating the expected
value of the trace of a high power of the adjacency matrix of a
random graph. To estimate this expected value, Linial and Puder (in
\cite{LP10}) study \emph{word maps} associated with formal words
over the alphabet $\S = \S_k = \{{g_1}^{\pm1}, \ldots,
{g_k}^{\pm1}\}$. Let $\S^*$ denote the set of all finite words in
the alphabet $\S$. Given a word $w \in \S^*$, the \emph{word map}
associated with $w$ maps the $k$-tuple $(\sigma_1, \ldots, \sigma_k)
\in {S_n}^k$ to the permutation $w(\sigma_1, \ldots, \sigma_k) \in
S_n$, where $w(\sigma_1, \ldots, \sigma_k)$ is the permutation
obtained by replacing $g_1, \ldots, g_k$ with $\sigma_1, \ldots,
\sigma_k$ (respectively) in the expression for $w$.

The results of \cite{BS87}, \cite{Fri03}, and \cite{LP10} all
involve studying the probability that 1 (or any given point in $\{1,
\ldots, n\}$) is a fixed point of the permutation $w(\sigma_1,
\ldots, \sigma_k)$, when $\sigma_1, \ldots, \sigma_k \in S_n$ are
chosen randomly with uniform distribution. We are interested in how
close this probability is to the corresponding probability in the
case of a random permutation, i.e. to $\frac{1}{n}$.

Formally, following the notation of \cite{LP10}, for every
$w\in\S^*$ and $n\in\N$ we denote by $X_w^{(n)}$ a random variable
on $S_n^{\;k}$ which is defined by:
\begin{equation} \label{eq:xwn}
X_w^{(n)}(\sigma_{1},\ldots,\sigma_{k}) = \textrm{\# of fixed points
of } w(\sigma_{1}, \ldots , \sigma_{k}).
\end{equation} \\

We then have $\Phi_w(n)$ defined as $\Phi_w(n) =
\frac{\E(X_w^{(n)})-1}{n} = \frac{\E(X_w^{(n)})}{n} - \frac{1}{n}$
(where we always assume the uniform distribution on $S_n^{\;k}$).
$\frac{\E(X_w^{(n)})}{n}$ is the probability that 1 (or any given
point in $\{1, \ldots, n\}$) is a fixed point of the permutation
$w(\sigma_1, \ldots, \sigma_k)$. Thus, $\Phi_w(n)$ measures how much
this probability differs from $\frac{1}{n}$ for the word $w$. In a
paper from 94', Nica \cite{Nic94} showed that for a fixed word $w$
and large enough $n$, $\Phi_w(n)$ can be expressed as a rational
function in $n$ of degree $\leq 0$.

To study $\Phi_w(n)$, Linial and Puder (in \cite{LP10}) formulated
two new and separate categorizations of formal words, namely $\phi,
\beta : \S^* \to \Z_{\geq 0} \cup \{\infty\}$, which are invariant
under reduction of words. Thus, $\phi$ and $\beta$ are also
categorizations of the words in the free group $\F = \F_k$ generated
by $\{g_1, \ldots, g_k\}$. $\phi(w)$ is defined in accordance with
the degree of the rational function corresponding to $\Phi_w(n)$.
More specifically, it can be shown (see \cite{LP10}, Lemma 4) that
for every word $w \in \S^*$ and $n \geq |w|$, we have
\begin{equation} \label{eq:Phi}
\Phi_w(n) = \frac{\E(X_w^{(n)})}{n}-\frac{1}{n} = \sum_{i=0}^{\infty}a_i(w)\frac{1}{n^i}
\end{equation}
where the coefficients $a_i(w)$ are integers depending only on $w$,
and $|w|$ denotes the length of $w$. We then define:

\begin{equation} \label{eq:phi}
\phi(w) := \left\{ \begin{array}{ll} \textrm{the smallest integer
$i$ with }a_i(w) \neq 0 & \textrm{ if } \mathbb{E}(X_w^{(n)}) \not
\equiv 1 \\
\infty & \textrm{ if } \mathbb{E}(X_w^{(n)}) \equiv 1
\end{array} \right.
\end{equation} \\

Thus, $\phi(w)$ measures how much the above probability differs from
$\frac{1}{n}$ for the word $w$. The higher $\phi(w)$ is, the closer
the probability is to $\frac{1}{n}$.

On the other hand, $\beta(w)$ is defined combinatorially without
explicit reference to word maps and the symmetric group. (A review
of the definition of $\beta$ is given below in Section
\ref{sbs:beta}). Both $\phi$ and $\beta$ extend the dichotomy of
primitive vs. imprimitive words in $\F$ (recall that $1\neq w \in
\F$ is said to be \emph{imprimitive}, as an element of $\F$, if $w =
u^d$ for some $u \in \F$ and $d \geq 2$). In some sense, $\phi(w)$
and $\beta(w)$ can be thought of as quantifying the ``level of
primitivity'' of the word $w$.

In \cite{LP10}, Linial and Puder have conjectured that $\phi(w) =
\beta(w)$ for every word $w$. They have proven that $\phi(w) = 0$
iff $\beta(w) = 0$ iff $w$ reduces to the empty word, and $\phi(w) =
1$ iff $\beta(w) = 1$ iff $w$ is imprimitive as an element of $\F$.
These two facts also appear in \cite{BS87} and \cite{Fri03}, but not
in the explicit language of $\phi$ and $\beta$. Linial and Puder
have also made partial progress towards proving $\phi(w) = 2$ iff
$\beta(w) = 2$, which allowed them to obtain an improved eigenvalue
bound compared to the result in \cite{Fri03}. Furthermore, they have
run extensive numerical simulations, and the results suggest that
$\phi(w) = \beta(w)$ for every word $w$.

Friedman had conjectured that almost every $n$-lift $H \in L_n(G)$
of $G$ satisfies $\mu_{max}(H) \leq \rho + o_n(1)$ (e.g., see
\cite{Fri03,HLW06}). It is known that every $n$-lift $H \in L_n(G)$
of $G$ satisfies $\mu_{max}(H) \geq \rho - o_n(1)$ (see
\cite{Gre95,Fri03,HLW06}), so one cannot prove a significantly
stronger result. The conjecture that $\phi(w) = \beta(w)$, if proven
true, gives us a very promising approach to proving a slightly weaker
version of Friedman's conjecture, namely $\mu_{max}(H) \leq
O(\rho^{1-\epsilon}\lambda_1^\epsilon)$ almost surely for every
$\epsilon>0$ (see \cite{LP10}). Also, if proven true, the conjecture
may also significantly simplify the usage of the Trace Method in
proving new (or old) eigenvalue bounds in various contexts.

Our work mainly builds on the work of Linial and Puder in
\cite{LP10}, and our main result is

\begin{theorem}\label{ther:phi=beta=2}
$\phi(w) = 2 \Leftrightarrow \beta(w) = 2$ for every word $w \in \S^*$
\end{theorem}

The $\phi(w) = \beta(w)$ conjecture is also interesting in other
aspects, such as its connection to word maps. For example, a
slightly stronger version of this conjecture made in \cite{LP10}
states that the first non-vanishing $a_i(w)$ in \eqref{eq:Phi} is in
fact positive. This yields that for every word $w$ and sufficiently
large $n$, $w(\sigma_1,\ldots,\sigma_k)$ has at least one fixed
point on average. A first result in this direction can be easily
inferred from \cite{Nic94}, where it is indirectly shown that
$(a_0(w),a_1(w))\geq(0,0)$ lexicographically, whence $\E(X_w^{(n)})
\geq 1-O(\frac{1}{n}$). Our proof of Theorem \ref{ther:phi=beta=2}
actually shows that whenever $a_0(w)=0$ and $a_1(w)=0$, we have
$a_2(w)\geq0$, and thus obtain:

\begin{corollary}\label{cor:E>1-1/n^2}
For every fixed word $w\in \S^*$,
\begin{displaymath}
\E(X_w^{(n)}) \geq 1-O\left(\frac{1}{n^2}\right)
\end{displaymath}
\end{corollary}

For those already familiar with the paper \cite{LP10}, we briefly
mention some aspects of our proof. For a word $w \in \S^*$ such that
$\beta(w) \geq 2$, recall that there exists a natural surjective
function $f$ from the (connected) components of some graph
$\Upsilon$ to the quotients in $\Q_w$ that have characteristic 2 and
type A (see \cite{LP10}). Linial and Puder believed that this
function is also injective, which would show that $\phi(w) = 2$ iff
$\beta(w) = 2$. In this paper, we follow this strategy and prove
that $f$ is indeed injective. This is done by recursively factoring
the word $w$ into finer and finer pieces, which allows us to analyze
the quotient graphs in $\Q_w$.

In Section \ref{sec:review} we review the definition of $\b(\cdot)$
and the established connections between $\b(\cdot)$ and $\f(\cdot)$.
In Section \ref{sec:proof}, after restoring the construction of the
aforementioned function $f$ (which first appeared in \cite{LP10}),
we show in Subsection \ref{sbs:injectivity} that it is indeed
injective, thereby proving Theorem \ref{ther:phi=beta=2} and
Corollary \ref{cor:E>1-1/n^2}.

\section{Review of \texorpdfstring{$\phi$}{phi} and of \texorpdfstring{$\beta$}{beta}} \label{sec:review}

In this section, we review some concepts and terminology from
\cite{LP10}, and also introduce some new terminology for
convenience.

\subsection{The Definition of \texorpdfstring{$\beta$}{beta}}
\label{sbs:beta}

We begin by describing some terminology needed to define $\beta(w)$
for any word $w \in \S^*$. Fix any word $w =
g_{i_1}^{\alpha_1}g_{i_2}^{\alpha_2}\ldots
g_{i_{|w|}}^{\alpha_{|w|}} \in \S^*$, where $i_1, i_2, \ldots,
i_{|w|} \in \{1, 2, \ldots, k\}$ and $\alpha_1, \alpha_2, \ldots,
\alpha_{|w|} \in \{-1, 1\}$. We denote by $T_w$ a path graph,
consisting of $|w|+1$ labeled vertices and $|w|$ directed and
labeled edges. The vertices are labeled $s_0,\ldots,s_{|w|}$. For
each $1\leq j \leq |w|$ there is an edge connecting $s_{j-1}$ and
$s_j$, labeled $i_j$ and directed according to $\alpha_j$. We shall
call the resulting graph the \emph{open trail} of $w$. E.g., for the
word $w = g_1g_2g_2{g_2}^{-1}g_3g_2g_1^{-1}$, the open trail $T_w$
of $w$ is the following directed edge-labeled graph:

\[\xy
(0,0)*+{s_0}="s0";
(18,0)*+{s_1}="s1";
(36,0)*+{s_2}="s2";
(54,0)*+{s_3}="s3";
(72,0)*+{s_4}="s4";
(90,0)*+{s_5}="s5";
(108,0)*+{s_6}="s6";
(126,0)*+{s_7}="s7";
{\ar^{1} "s0";"s1"};
{\ar^{2} "s1";"s2"};
{\ar^{2} "s2";"s3"};
{\ar_{2} "s4";"s3"};
{\ar^{3} "s4";"s5"};
{\ar^{2} "s5";"s6"};
{\ar_{1} "s7";"s6"};
\endxy
\]

The definition of $\beta(w)$ and analysis of $\phi(w)$ both rely on
the notion of quotient graphs of $T_w$. A quotient graph $\G$ of
$T_w$ corresponds to a partition of  $s_0,\ldots,s_{|w|}$. The
vertices of $\G$ correspond to the blocks of the partition, and
there is a $j$-labeled directed edge from the block $U$ to the block
$V$ whenever there is some $s_h\in U$ and $s_k\in V$ such that there
is a $j$-labeled edge in $T_w$ from $s_h$ to $s_k$. We write $s_h
\stackrel[\G]{}{\equiv} s_l$ (or simply $s_h \equiv s_l$ if $\G$ is
clear from the context) whenever $s_h$ and $s_l$ belong to the same
block in the partition corresponding to $\G$.

Among all quotient graphs of $T_w$ we are interested in those
satisfying two conditions. We first demand that the trail of $w$ in
the quotient be closed (the importance of this will be clear in the
analysis of $\phi(w)$). We further demand that no two $j$-labeled
edges share the same origin or the same terminus. This, in turn,
will guarantee that if we focus on paths in the quotient that start
in some fixed vertex, then any two different paths will correspond
to different words in $\S^*$. Formally, we define

\begin{definition} \label{def:realizable}
A \textbf{realizable quotient graph} of $T_w$ is a quotient graph
(or, equivalently, a partition of the set $\{s_0,\ldots,s_{|w|}\}$)
such that the following conditions hold:
 \begin{enumerate}
   \item $s_0 \equiv s_{|w|}$
   \item Whenever $i_h=i_l$ and $\alpha_h=\alpha_l$, we have $s_{h-1}\equiv s_{l-1} \iff s_h\equiv s_l$
   \item Whenever $i_h=i_l$ and $\alpha_h=-\alpha_l$, we have $s_{h-1}\equiv s_l \iff s_h\equiv s_{l-1}$
 \end{enumerate}
 We denote by $\Q_w$ the set of all realizable quotient graphs of $T_w$.
\end{definition}

To illustrate, we draw (Figure \ref{fig:q_w}) all the realizable
quotient graphs of the commutator word in $\F_2$.
\begin{figure}[h]
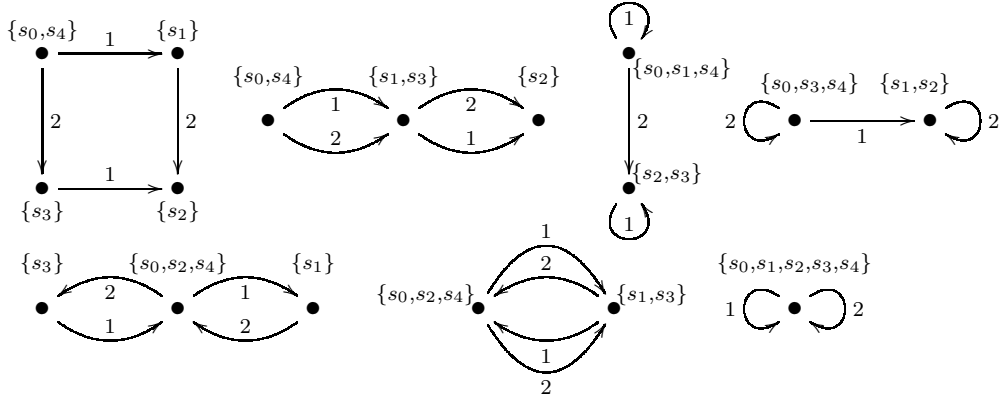

\begin{center}
\xy
(0,40)*+{\bullet}="s0"+(0,3)*{\scriptstyle \{s_0,s_4\}};
(18,40)*+{\bullet}="s1"+(0,3)*{\scriptstyle \{s_1\}};
(18,22)*+{\bullet}="s2"+(0,-3)*{\scriptstyle \{s_2\}};
(0,22)*+{\bullet}="s3"+(0,-3)*{\scriptstyle \{s_3\}};
{\ar^{1} "s0";"s1"};
{\ar^{2} "s1";"s2"};
{\ar^{1} "s3";"s2"};
{\ar^{2} "s0";"s3"};
%
(30,31)*+{\bullet}="r0"+(0,6)*{\scriptstyle \{s_0,s_4\}};
(48,31)*+{\bullet}="r1"+(0,6)*{\scriptstyle \{s_1,s_3\}};
(66,31)*+{\bullet}="r2"+(0,6)*{\scriptstyle \{s_2\}};
{\ar@/^1pc/_{1} "r0";"r1"};
{\ar@/^1pc/_{2} "r1";"r2"};
{\ar@/_1pc/^{1} "r1";"r2"};
{\ar@/_1pc/^{2} "r0";"r1"};
%
(78,40)*+{\bullet}="t0"+(7,-2)*{\scriptstyle \{s_0,s_1,s_4\}};
(78,22)*+{\bullet}="t1"+(5,2)*{\scriptstyle \{s_2,s_3\}};
{\ar@(ul,ur)_{1} "t0";"t0"};
{\ar^{2} "t0";"t1"};
{\ar@(dl,dr)^{1} "t1";"t1"};
%
(100,31)*+{\bullet}="u0"+(2,5)*{\scriptstyle \{s_0,s_3,s_4\}};
(118,31)*+{\bullet}="u1"+(-2,5)*{\scriptstyle \{s_1,s_2\}};
{\ar@(ul,dl)_{2} "u0";"u0"};
{\ar_{1} "u0";"u1"};
{\ar@(ur,dr)^{2} "u1";"u1"};
%
(0,6)*+{\bullet}="v0"+(0,6)*{\scriptstyle \{s_3\}};
(18,6)*+{\bullet}="v1"+(0,6)*{\scriptstyle \{s_0,s_2,s_4\}};
(36,6)*+{\bullet}="v2"+(0,6)*{\scriptstyle \{s_1\}};
{\ar@/^1pc/_{1} "v1";"v2"};
{\ar@/^1pc/_{2} "v2";"v1"};
{\ar@/_1pc/^{1} "v0";"v1"};
{\ar@/_1pc/^{2} "v1";"v0"};
%
(58,6)*+{\bullet}="w0"+(-7,2)*{\scriptstyle \{s_0,s_2,s_4\}};
(76,6)*+{\bullet}="w1"+(5,2)*{\scriptstyle \{s_1,s_3\}};
{\ar@/^2pc/^{1} "w0";"w1"};
{\ar@/_1pc/_{2} "w1";"w0"};
{\ar@/^1pc/^{1} "w1";"w0"};
{\ar@/_2pc/_{2} "w0";"w1"};
%
(100,6)*+{\bullet}="x"+(0,6)*{\scriptstyle \{s_0,s_1,s_2,s_3,s_4\}};
{\ar@(ul,dl)_{1} "x";"x"};
{\ar@(ur,dr)^{2} "x";"x"};
\endxy
\end{center}
\caption{The set $\Q_w$ of realizable quotient graphs when $w=g_1g_2g_1^{-1}g_2^{-1}\in\F_2$.} \label{fig:q_w}
\end{figure}

We next concentrate on the number of pairs of $s_i$'s that should be
merged in order to yield a specific quotient graph $\G\in \Q_w$. We
say that $\G$ \emph{is generated by the set of pairs}
$\{\{s_{j_1},s_{k_1}\},\ldots,\{s_{j_r},s_{k_r}\}\}$ if the
partition corresponding to $\G$ is the finest partition in which
$s_{j_i}\equiv s_{k_i}~ \forall i=1,\ldots,r$, and such that no two
$j$-labeled edges share the same origin or the same terminus.
Equivalently, we can generate the quotient obtained from a certain
set $E$ of pairs by making gradually all necessary merges, and only
them: start with the partition
$\{\{s_0\},\{s_1\},\ldots,\{s_{|w|}\}\}$ and gradually merge every
two blocks that contain the two elements of the same pair in $E$.
Then gradually merge every two blocks which are the origin
(terminus) of two $j$-labeled edge with the same terminus (resp.
origin). It is not hard to see that the order of merging has no
significance. To illustrate, we show in Figure
\ref{fig:generating_set_of_pairs} how we obtain the upper right
quotient from Figure \ref{fig:q_w} from the set of pairs
$\{\{s_0,s_3\},\{s_0,s_4\}\}$:

\begin{figure}[h]
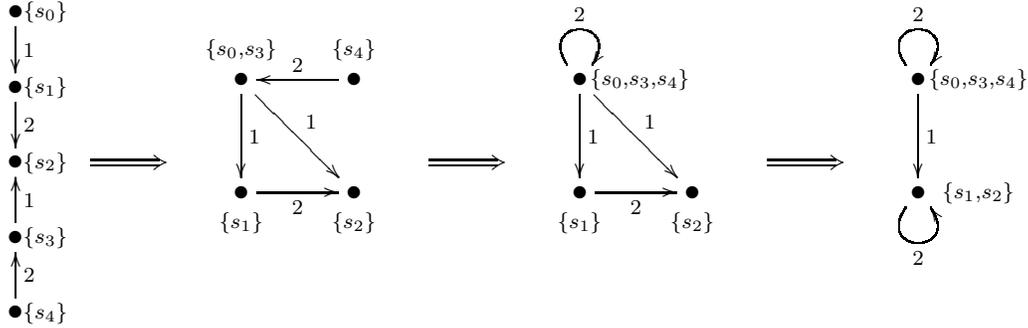

\begin{center}
\xy
(0,40)*+{\bullet}="s0"+(4,0)*{\scriptstyle \{s_0\}};
(0,30)*+{\bullet}="s1"+(4,0)*{\scriptstyle \{s_1\}};
(0,20)*+{\bullet}="s2"+(4,0)*{\scriptstyle \{s_2\}};
(0,10)*+{\bullet}="s3"+(4,0)*{\scriptstyle \{s_3\}};
(0,0)*+{\bullet}="s4"+(4,0)*{\scriptstyle \{s_4\}};
{\ar^{1} "s0";"s1"};
{\ar^{2} "s1";"s2"};
{\ar_{1} "s3";"s2"};
{\ar_{2} "s4";"s3"};
{\ar@{=>} (10,20)*{}; (20,20)*{}};
(30,31)*+{\bullet}="r0"+(0,4)*{\scriptstyle \{s_0,s_3\}};
(30,16)*+{\bullet}="r1"+(0,-4)*{\scriptstyle \{s_1\}};
(45,16)*+{\bullet}="r2"+(0,-4)*{\scriptstyle \{s_2\}};
(45,31)*+{\bullet}="r4"+(0,4)*{\scriptstyle \{s_4\}};
{\ar^{1} "r0";"r1"};
{\ar_{2} "r1";"r2"};
{\ar^{1} "r0";"r2"};
{\ar_{2} "r4";"r0"};
{\ar@{=>} (55,20)*{}; (65,20)*{}};
(75,31)*+{\bullet}="t0"+(8,0)*{\scriptstyle \{s_0,s_3,s_4\}};
(75,16)*+{\bullet}="t1"+(0,-4)*{\scriptstyle \{s_1\}};
(90,16)*+{\bullet}="t2"+(0,-4)*{\scriptstyle \{s_2\}};
{\ar^{1} "t0";"t1"};
{\ar_{2} "t1";"t2"};
{\ar^{1} "t0";"t2"};
{\ar@(ul,ur)^{2} "t0";"t0"};
{\ar@{=>} (100,20)*{}; (110,20)*{}};
%
(120,31)*+{\bullet}="u0"+(8,0)*{\scriptstyle \{s_0,s_3,s_4\}};
(120,16)*+{\bullet}="u1"+(8,0)*{\scriptstyle \{s_1,s_2\}};
{\ar^{1} "u0";"u1"};
{\ar@(dl,dr)_{2} "u1";"u1"};
{\ar@(ul,ur)^{2} "u0";"u0"};
\endxy
\end{center}
\caption{Obtaining the quotient graph of $w=g_1g_2g_1^{-1}g_2^{-1}$
which is generated by the set of pairs
$\{\{s_0,s_3\},\{s_0,s_4\}\}$: We first merge the blocks containing
$s_0$ and $s_3$ (the first generating pair), and then the blocks
containing $s_0$ and $s_4$ (the second pair). We finish by merging
the blocks $\{s_1\}$ and $\{s_2\}$ because they are the termini of
$1$-edges which share the same origin.}
\label{fig:generating_set_of_pairs}
\end{figure}

For every graph $\G$ we denote by $\chi(\G) = e_\G - v_\G + 1$ the
Euler characteristic of $\G$. It turns out (\cite{LP10}, Lemma 6)
that the smallest cardinality of a generating set of pairs of $\G
\in \Q_w$ is $\chi(\G)$. We analyze the smallest generating sets of
each $\G\in\Q_w$ and define:

\begin{definition} \label{def:type_a_b}
Let $w$ be a word in $\S^*$. We say that a quotient graph
$\G\in\Q_w$ \textbf{has type A}, if one of the smallest generating
sets of pairs for $\G$ contains the pair $\{s_0,s_{|w|}\}$.
Otherwise, we say $\G$ \textbf{has type B}.
\end{definition}

For example, out of the seven quotient graphs in Figure
\ref{fig:q_w}, only the figure-eight graph with one vertex and two
edges has type B. The other six graphs have type A, as we show in
Figure \ref{fig:generating_sets}.

\begin{figure}[h]
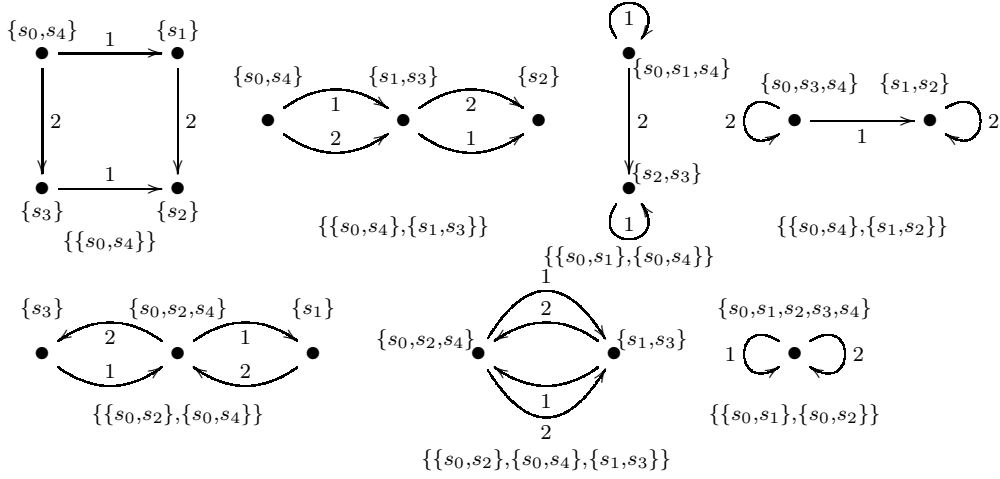

\begin{center}
\xy
(0,40)*+{\bullet}="s0"+(0,3)*{\scriptstyle \{s_0,s_4\}};
(18,40)*+{\bullet}="s1"+(0,3)*{\scriptstyle \{s_1\}};
(18,22)*+{\bullet}="s2"+(0,-3)*{\scriptstyle \{s_2\}};
(0,22)*+{\bullet}="s3"+(0,-3)*{\scriptstyle \{s_3\}};
{\ar^{1} "s0";"s1"};
{\ar^{2} "s1";"s2"};
{\ar^{1} "s3";"s2"};
{\ar^{2} "s0";"s3"};
(9,15)*{\scriptstyle \{\{s_0,s_4\}\}};
%
(30,31)*+{\bullet}="r0"+(0,6)*{\scriptstyle \{s_0,s_4\}};
(48,31)*+{\bullet}="r1"+(0,6)*{\scriptstyle \{s_1,s_3\}};
(66,31)*+{\bullet}="r2"+(0,6)*{\scriptstyle \{s_2\}};
{\ar@/^1pc/_{1} "r0";"r1"};
{\ar@/^1pc/_{2} "r1";"r2"};
{\ar@/_1pc/^{1} "r1";"r2"};
{\ar@/_1pc/^{2} "r0";"r1"};
(48,17)*{\scriptstyle \{\{s_0,s_4\},\{s_1,s_3\}\}};
%
(78,40)*+{\bullet}="t0"+(7,-2)*{\scriptstyle \{s_0,s_1,s_4\}};
(78,22)*+{\bullet}="t1"+(5,2)*{\scriptstyle \{s_2,s_3\}};
{\ar@(ul,ur)_{1} "t0";"t0"};
{\ar^{2} "t0";"t1"};
{\ar@(dl,dr)^{1} "t1";"t1"};
(78,13)*{\scriptstyle \{\{s_0,s_1\},\{s_0,s_4\}\}};
%
(100,31)*+{\bullet}="u0"+(2,5)*{\scriptstyle \{s_0,s_3,s_4\}};
(118,31)*+{\bullet}="u1"+(-2,5)*{\scriptstyle \{s_1,s_2\}};
{\ar@(ul,dl)_{2} "u0";"u0"};
{\ar_{1} "u0";"u1"};
{\ar@(ur,dr)^{2} "u1";"u1"};
(109,17)*{\scriptstyle \{\{s_0,s_4\},\{s_1,s_2\}\}};
%
(0,0)*+{\bullet}="v0"+(0,6)*{\scriptstyle \{s_3\}};
(18,0)*+{\bullet}="v1"+(0,6)*{\scriptstyle \{s_0,s_2,s_4\}};
(36,0)*+{\bullet}="v2"+(0,6)*{\scriptstyle \{s_1\}};
{\ar@/^1pc/_{1} "v1";"v2"};
{\ar@/^1pc/_{2} "v2";"v1"};
{\ar@/_1pc/^{1} "v0";"v1"};
{\ar@/_1pc/^{2} "v1";"v0"};
(18,-8)*{\scriptstyle \{\{s_0,s_2\},\{s_0,s_4\}\}};
%
(58,0)*+{\bullet}="w0"+(-7,2)*{\scriptstyle \{s_0,s_2,s_4\}};
(76,0)*+{\bullet}="w1"+(5,2)*{\scriptstyle \{s_1,s_3\}};
{\ar@/^2pc/^{1} "w0";"w1"};
{\ar@/_1pc/_{2} "w1";"w0"};
{\ar@/^1pc/^{1} "w1";"w0"};
{\ar@/_2pc/_{2} "w0";"w1"};
(67,-14)*{\scriptstyle \{\{s_0,s_2\},\{s_0,s_4\},\{s_1,s_3\}\}};
%
(100,0)*+{\bullet}="x"+(0,6)*{\scriptstyle \{s_0,s_1,s_2,s_3,s_4\}};
{\ar@(ul,dl)_{1} "x";"x"};
{\ar@(ur,dr)^{2} "x";"x"};
(100,-8)*{\scriptstyle \{\{s_0,s_1\},\{s_0,s_2\}\}};
\endxy
\end{center}
\caption{Smallest generating sets for the quotients graphs in $\Q_w$
when $w=g_1g_2g_1^{-1}g_2^{-1}\in\F_2$. Note that the size of the
generating set of each quotient graph $\G$ equals $\chi(\G)$. All
graphs except for the bottom right one have type A: they have a
smallest generating set that contains $\{s_0,s_4\}$. The remaining
figure-eight graph has type B: none of his smallest generating sets
contains $\{s_0,s_4\}$.} \label{fig:generating_sets}
\end{figure}

We now have all the ingredients required to define $\beta(w)$:
\begin{definition} \label{def:beta}
Let $w$ be a word in $\S^*$. We define \textbf{$\beta(w)$} to be the
smallest characteristic of a type-B graph in $\Q_w$. Namely,
\begin{displaymath}
\beta(w) := \min\left( \{ \chi(\Gamma)~:~ \Gamma \in {\cal
Q}_w \textrm{ and $\Gamma$ has type B} \} \cup \{\infty\} \right)
\end{displaymath}
\end{definition}

For example, for the commutator word $w=g_1g_2g_1^{-1}g_2^{-1}$ we have $\beta(w)=2$.\\

Lemma 9 of \cite{LP10} shows that $\b (\cdot)$ is invariant under
reduction of words, so that it is in fact a well defined function on
the free group $\F_k$. The preceding lemma therein shows that $\b
(\cdot)$ is also invariant under cyclic shift, and in fact it can be
shown that it is invariant under the action of $Aut(\F_k)$ on
$\F_k$.

\subsection{An Algorithm to Calculate \texorpdfstring{$\phi$}{phi}} \label{sbs:phi}
The source for the strong relations between $\b (\cdot)$ and $\f
(\cdot)$ is in the role played by the set $\Q_w$ in an algorithm to
calculate $\f(w)$ (in addition to its role in the definition of
$\b(w)$). This algorithm was initiated by \cite{Nic94} and further
analyzed in \cite{LP10}.

Fix some $w\in\S^*$. In order to calculate the probability that $1$
is a fixed point of the permutation $w(\s_1,\ldots,\s_k)\in S_n$ (as
usual, $\s_1,\ldots,\s_k$ are randomly chosen from $S_n$), we trace
the trail of $1$ through $w(\s_1,\ldots,\s_k)$. If $w =
g_{i_1}^{\alpha_1}g_{i_2}^{\alpha_2}\ldots
g_{i_{|w|}}^{\alpha_{|w|}} \in \S^*$, we first look at
$s_1=\s_{i_1}^{\alpha_1}(1)\in[n]$, then at
$s_2=\s_{i_2}^{\alpha_2}\left(\s_{i_1}^{\alpha_1}(1)\right)\in[n]$,
etc. For instance, for $w=g_1g_2g_1^{-1}g_2^{-1}$ we draw the trail
of $1$ as follows:

\[
\xy
(20,0)*+{1}="s0";
(40,0)*+{s_1}="s1";
(60,0)*+{s_2}="s2";
(80,0)*+{s_3}="s3";
(100,0)*+{s_4}="s4";
{\ar^{1} "s0";"s1"};
{\ar^{2} "s1";"s2"};
{\ar_{1} "s3";"s2"};
{\ar_{2} "s4";"s3"};
\endxy
\]

Note that for the sake of convenience, we compose permutations from
left to right. This is inconsequential for the analysis of the
variables $X_w^{(n)}$ and the function $\f(w)$ since $w(\sigma_1,\ldots,\sigma_k)$ with
left-to-right composition is the inverse of
$w(\sigma_1^{\;-1},\ldots,\sigma_k^{\;-1})$ with right-to-left
composition, and thus both have the same cycle structure.

Continuing with the example of $w=g_1g_2g_1^{-1}g_2^{-1}$, whenever
$w(\s_1,\s_2)$ fixes 1 it is obviously always the case that $s_4=1$.
We then divide this event to several disjoint sub-events according
to the ``coincidence pattern'' of $1,s_1,s_2$ and $s_3$. For
instance, we have the subevent where $s_1,s_2$ and $s_3$ consist of
three distinct numbers in $\{2,3,\ldots,n\}$. This event corresponds
to the quotient graph

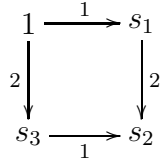
\[
\xy
(0,15)*+{1}="s0";
(15,15)*+{s_1}="s1";
(15,0)*+{s_2}="s2";
(0,0)*+{s_3}="s3";
{\ar^{1} "s0";"s1"};
{\ar^{2} "s1";"s2"};
{\ar_{1} "s3";"s2"};
{\ar_{2} "s0";"s3"};
\endxy
\]
and its probability is $\frac{(n-1)(n-2)(n-3)}{n(n-1)\cdot n(n-1)}$:
The number of possible values of $s_1,s_2,s_3$ is $(n-1)(n-2)(n-3)$.
The chance that $\s_1(1)=s_1$ and $\s_1(s_3)=s_2$ is
$\frac{1}{n(n-1)}$, and likewise the probability that
$\s_2(s_1)=s_2$ and $\s_2(1)=s_3$ equals $\frac{1}{n(n-1)}$ (we
assume $n\geq2$ for both calculations). A different subevent
corresponds to the case where $s_1=s_2$ and $s_3=1$, and is depicted
by the graph

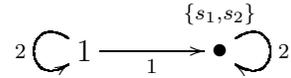
\[\xy
(0,0)*+{1}="u0";
(18,0)*+{\bullet}="u1"+(0,5)*{\scriptstyle \{s_1,s_2\}};
{\ar@(ul,dl)_{2} "u0";"u0"};
{\ar_{1} "u0";"u1"};
{\ar@(ur,dr)^{2} "u1";"u1"};
\endxy\]
By similar arguments we get that the probability of this subevent is $\frac{n-1}{n\cdot n(n-1)}$. \\

Note that not all coincidence patterns are realizable. For instance, it is impossible to have $s_3=1$ and $s_1\neq s_2$, because $s_1=\s_1(1)$ whereas $s_2=\s_1(s_3)$. These considerations and the requirement that $s_4$ equal $1$ show that the realizable coincidence patterns are exactly those corresponding to realizable quotient graphs of $w$, i.e. to graphs in the set $\Q_w$ introduced in Definition \ref{def:realizable}.

In general, for each $\G\in\Q_w$, the probability that $\G$ depicts the trail of $1$ through $w(\s_1,\ldots,\s_k)$ equals
\[
\frac{(n-1)(n-2)\ldots(n-v_\G+1)}{\prod_{j=1}^k{n(n-1)\ldots(n-e_\G^j+1)}},
\]
where $v_\G$ is the number of vertices in $\G$, $e_\G^j$ is the number of $j$-labeled edges in $\G$ and $n$ is assumed to be $\geq e_\G^j~ \forall j$. This shows that for $n\geq e_\G^j~ \forall j$, $\Phi_w(n)$ can be calculated as follows:
\begin{eqnarray} \label{eq:Phi_formula}
\Phi_w(n) &=& Prob\left(w(\s_1,\ldots,\s_k)(1)=1\right) - \frac{1}{n} \nonumber \\
&=& - \frac{1}{n} + \sum_{\G\in\Q_w}{\frac{(n-1)(n-2)\ldots(n-v_\G+1)}{\prod_{j=1}^k{n(n-1)\ldots(n-e_\G^j+1)}}}
\end{eqnarray}

Thus $\Phi_w$ is a rational function in $n$ for $n$ large enough. Note that the degree of each term in the summation in \eqref{eq:Phi_formula} is $-\chi(\G)\leq0$ (the Euler characteristic is non-negative for connected graphs), so the degree of $\Phi_w$ as a rational function in $n$ in non-positive. This shows that we can write this rational function as $\sum_{i=0}^\infty{\frac{a_i(w)}{n^i}}$, as mentioned in \eqref{eq:Phi}. $\phi(w)$ was defined to be the smallest $i$ for which $a_i(w)$ does not vanish, which is exactly the additive inverse of the degree of this rational function.

Note also that for each $i$, $a_i(w)$ is \emph{only affected by the quotients graphs $\G\in\Q_w$ with $\chi(\G)\leq i$}. Thus, when analyzing the coefficients $a_0(w),a_1(w)$ and $a_2(w)$, one has to analyze only the quotients graphs of characteristic $\leq 2$.

Finally, it is easy to see that $\f(\cdot)$ is invariant under reduction of words (by its definition) and so, like $\b(\cdot)$ is a well defined function on the free group $\F_k$. Moreover, if $\psi\in Aut(\F_k)$ then for each $w\in\F_k$, $w$ and $\psi(w)$, acting as word maps, induce the same distribution on $S_n$. So like $\b(\cdot)$ again, $\f(\cdot)$ in invariant under the action of $Aut(\F_k)$ on $\F_k$.

\subsection{Already Established Relations Between \texorpdfstring{$\phi$}{phi} and \texorpdfstring{$\beta$}{beta}} \label{sbs:know_relations}
Before moving forward to proving that $\phi(w)=2 \iff \beta(w)=2$, we would like to give a short summary of the connections between the two functions that were already established in \cite{LP10}.

To begin with, $\f(w)=i \iff \b(w)=i$ for $i=0,1$ (Lemmas 12 and 13 in \cite{LP10}). The case $i=0$ is a bit degenerate: it occurs only when there is a quotient graph of Euler characteristic $0$ in $\Q_w$. It means that with an \emph{empty} generating set of pairs we obtain a quotient graph where $s_0\equiv s_{|w|}$. This only happens when $w$ reduces to the empty word, i.e. $w=1$ as an element of $\F_k$.

The case $i=1$ is more interesting. We can assume that $w\neq 1$ (in $\F_k$), and so there are no quotient graphs of characteristic $0$. Since every quotient graph of characteristic $1$ contributes exactly $1$ to $a_1(w)$ and we subtract $\frac{1}{n}$ from the sum in \eqref{eq:Phi_formula}, this coefficient equals the number of $\G\in\Q_w$ with $\chi(\G)=1$, minus 1. Since there is exactly one type-A graph in $\Q_w$ with characteristic $1$ (the one generated by $\{\{s_0,s_{|w|}\}\}$), we get that $a_1(w)$ equals the number of type-B quotient graphs of characteristic $1$, so that indeed $\phi(w)=1 \iff \b(w)=1$. In fact, as explained in \cite{LP10}, this case corresponds exactly to the case where $w$ is a power of another word (i.e. $w=u^d$ for some $u\in\F_k$ and $d\geq 2$).

Another interesting connection between the two functions occurs when $w$ is the single letter word $w=g_1$. This word induces the uniform distribution on $S_n$ (and in fact on any finite group) when applied as a word map from $S_n^{\;k}$ to $S_n$, and so $\f(w)=\infty$. But $\Q_w$ contains only one graph, and it is easy to verify that we also have $\b(w)=\infty$. Since both functions are invariant under the action of $Aut(\F_k)$ on $\F_k$, we get that they both agree and equal $\infty$ on the entire orbit of the single letter word in $\F_k$. (This orbit includes, for example, the word $w=g_1g_2g_1g_2g_1$.)

\section{Proof of the Main Theorem} \label{sec:proof}
In this section we aim to prove our main theorem, namely that for every $w\in\F_k$, $\f(w)=2 \iff \b(w)=2$. We follow here the same path of partial proof set by \cite{LP10}, and complete it to obtain Theorem \ref{ther:phi=beta=2}.

Fix some $w=g_{i_1}^{\alpha_1}g_{i_2}^{\alpha_2}\ldots g_{i_{|w|}}^{\alpha_{|w|}} \in\S^*$ (where $i_1, i_2, \ldots, i_{|w|} \in \{1, 2, \ldots, k\}$ and $\alpha_1, \alpha_2, \ldots, \alpha_{|w|} \in \{-1, 1\}$). Along the proof we have several simplifying assumptions on $w$. As we already know that $\f(w)=i \iff \b(w)=i$ for $i=0,1$, we assume that $\b(w),\f(w)\geq2$. (Equivalently, we assume that as an element in $\F_k$, $w\neq1$, nor is it a power of another word.) In addition, it is explicitly shown in \cite{LP10} that $\f(\cdot)$ and $\b(\cdot)$ are invariant under cyclic reduction, so we can also assume that $w$ is cyclically reduced, i.e. that $g_{i_j}^{\alpha_j}g_{i_{j+1}}^{\alpha_{j+1}}\neq1$ ($j=1,2,\ldots,|w|-1$), as well as that $g_{i_{|w|}}^{\alpha_{|w|}}g_{i_1}^{\alpha_1}\neq1$.\\

We continue by analyzing the coefficient $a_2(w)$ in the series in \eqref{eq:Phi}. Because $\f(w)\geq2$ we know that $a_0(w)=a_1(w)=0$. By the review in Section \ref{sec:review}, we also know that there are no quotient graphs $\G\in\Q_w$ with $\chi(\G)=0$ and that there is exactly one quotient graph with $\chi(\G)=1$ (the graph obtained by merging $s_0$ with $s_{|w|}$).

\begin{definition} \label{def:universal_graph}
Denote by $\Gtilde_w$ the quotient graph in $\Q_w$ generated by the set of pairs $\{\{s_0,s_{|w|}\}\}$. We call this graph \textbf{the universal graph} of $w$.
\end{definition}

Under our assumptions, $\chi(\Gtilde_w)=1$ and it has type-A. Because $w$ is cyclically reduced, $\Gtilde_w$ is a simple-circle graph, with exactly $|w|$ vertices and $|w|$ edges. For instance, the universal graph of the commutator word $w=g_1g_2g_1^{-1}g_2^{-1}$ is the upper left one in Figure \ref{fig:generating_sets}.

As we mentioned above, $a_2(w)$ is affected only by quotient graphs in $\Q_w$ with characteristic $\leq 2$. Moreover, each $\G\in\Q_w$ with $\chi(\G)=2$ contributes exactly $1$ to $a_2(w)$. Thus, under our assumptions, $a_2(w)$ consists of the contribution of $\Gtilde_w$ to it, plus the number of $\G\in\Q_w$ with $\chi(\G)=2$. We can therefore reduce to the following lemma which yields both Theorem \ref{ther:phi=beta=2} and Corollary \ref{cor:E>1-1/n^2}.

\begin{lemma} \label{lem:balancing}
Let $w\in \F_k$ have $\f(w)\geq2$ ($\iff \b(w)\geq2$). Then the contribution of $\Gtilde_w$ to $a_2(w)$ exactly balances off the contribution of all type-A quotient graphs of characteristic $2$. Put differently,
\[
a_2(w) = \left|\left\{\G\in\Q_w~:~\chi(\G)=2 \textrm{ and } \G \textrm{ has type-B}\right\}\right|
\]
\end{lemma}

Recall that the contribution of $\Gtilde_w$ to the summation in \eqref{eq:Phi_formula} is
\[
\frac{(n-1)(n-2)\ldots(n-v_{\Gtilde_w}+1)}{\prod_{j=1}^k{n(n-1)\ldots(n-e_{\Gtilde_w}^j+1)}} \]
and since $v_{\Gtilde_w}=e_{\Gtilde_w}$, a simple analysis shows that if we expand this to a power series in $\frac{1}{n}$, we obtain
\[
    \frac{1}{n} - \frac{\binom{v_{\Gtilde_w}}{2} - \sum_{j=1}^k{\binom{e_{\Gtilde_w}^j}{2}}}{n^2} + O\left(\frac{1}{n^3}\right)
\]
Our goal therefore reduces to showing that there are exactly $\binom{v_{\Gtilde_w}}{2} - \sum_{j=1}^k{\binom{e_{\Gtilde_w}^j}{2}}$ quotient graphs $\G\in\Q_w$ with $\chi(\G)=2$ and of type-A. We denote this subset of quotient graphs in $\Q_w$ by $\Q_{w,2,A}$:
\begin{definition} \label{def:Q_w2A}
Denote by $\Q_{w,2,A}$ the following subset of $\Q_w$:
\[
    \Q_{w,2,A} := \left\{ \G\in\Q_w~:~\chi(\G)=2~\wedge~type(\G)=A\right\}
\]
\end{definition}

And Lemma \ref{lem:balancing} then reduces to showing that under our simplifying assumptions,
\begin{equation} \label{eq:size_of_Qw2a}
\left|\Q_{w,2,A}\right| = \binom{v_{\Gtilde_w}}{2} - \sum_{j=1}^k{\binom{e_{\Gtilde_w}^j}{2}}
\end{equation}

Recall that, by definition, every $\G\in\Q_{w,2,A}$ can be generated by a set of two pairs, one of which is $\{s_0,s_{|w|}\}$. This is equivalent to saying that every such quotient graph is obtained from $\Gtilde_w$ by merging a single pair. (Indeed, it is an easy observation that the order in which we merge the pairs of a generating set has no significance, and we can obtain our quotient graph gradually, going through quotient graphs of smaller characteristic). In fact, we can (and will) view every $\G\in\Q_{w,2,A}$ as a realizable partition of the \emph{vertices of $\Gtilde_w$}, rather than of $\{s_0,\ldots,s_{|w|}\}$. (By realizable partition we mean here simply that in the resulting quotient graph no two $j$-edges share the same origin or the same terminus.)

This observation gives a first idea as to why the size of $\Q_{w,2,A}$ is indeed $\binom{v_{\Gtilde_w}}{2} - \sum_{j=1}^k{\binom{e_{\Gtilde_w}^j}{2}}$: The total number of pairs of vertices in $\Gtilde_w$ is $\binom{v_{\Gtilde_w}}{2}$, but clearly different pairs may generate the same $\G$. In particular, for any two $j$-edges in $\Gtilde_w$, the pair of origins generates the same quotient as the pair of termini. In the rest of the proof we will show that roughly, this is the only reason we get identical quotients.\\

To understand the full picture, we follow \cite{LP10} and introduce $\U$, ``the graph of pairs of vertices'' of $\Gtilde_w$. The graph $\U$ has $\binom{v_{\Gtilde_w}}{2}$ vertices labeled by pairs of vertices of $\Gtilde_w$, and has $\sum_{j=1}^k{\binom{e_{\Gtilde_w}^j}{2}}$ edges, one for each pair of same-color edges in $\Gtilde_w$. The edge corresponding to the pair $\{\epsilon_1,\epsilon_2\}$ of $j$-edges, is a
$j$-edge from the vertex $\{origin(\epsilon_1) ,origin(\epsilon_2)\}$
to $\{terminus(\epsilon_1) ,terminus(\epsilon_2)\}$. For example, when $w$ is the commutator word, $\U$ consists of $\binom{4}{2}=6$ vertices and $\binom{2}{2}+\binom{2}{2}=2$ edges. We illustrate a slightly more interesting case in Figure \ref{fig:Upsilon}.

\begin{figure}[h]
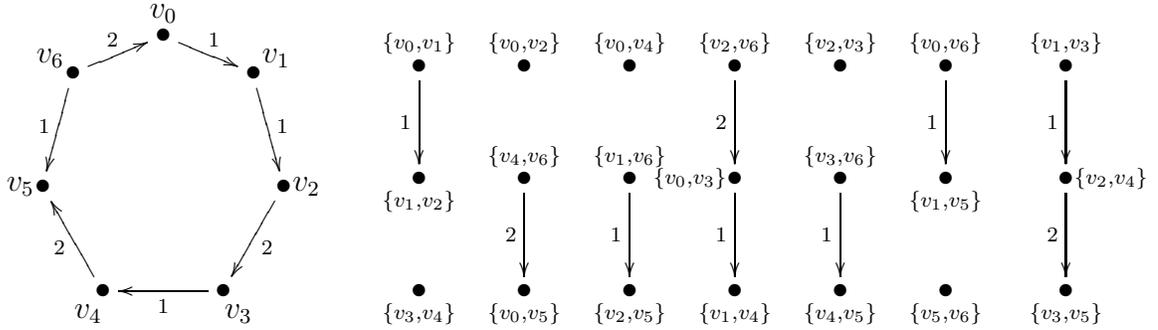

\begin{center}
\xy
(16,34)*+{\bullet}="v0"+(0,3)*{v_0};
(28,29)*+{\bullet}="v1"+(3,2)*{v_1};
(32,14)*+{\bullet}="v2"+(3,0)*{v_2};
(24,0 )*+{\bullet}="v3"+(2,-3)*{v_3};
(8 ,0 )*+{\bullet}="v4"+(-2,-3)*{v_4};
(0 ,14)*+{\bullet}="v5"+(-3,0)*{v_5};
(4 ,29)*+{\bullet}="v6"+(-3,2)*{v_6};
{\ar^{1} "v0";"v1"};
{\ar^{1} "v1";"v2"};
{\ar^{2} "v2";"v3"};
{\ar^{1} "v3";"v4"};
{\ar^{2} "v4";"v5"};
{\ar_{1} "v6";"v5"};
{\ar^{2} "v6";"v0"};
%
(50,30)*+{\bullet}="t0"+(0,3)*{\scriptstyle \{v_0,v_1\}};
(50,15)*+{\bullet}="t1"+(0,-3)*{\scriptstyle \{v_1,v_2\}};
(50,0 )*+{\bullet}="t2"+(0,-3)*{\scriptstyle \{v_3,v_4\}};
(64,30)*+{\bullet}="t3"+(0,3)*{\scriptstyle \{v_0,v_2\}};
(64,15)*+{\bullet}="t4"+(0,3)*{\scriptstyle \{v_4,v_6\}};
(64,0 )*+{\bullet}="t5"+(0,-3)*{\scriptstyle \{v_0,v_5\}};
(78,30)*+{\bullet}="t6"+(0,3)*{\scriptstyle \{v_0,v_4\}};
(78,15)*+{\bullet}="t7"+(0,3)*{\scriptstyle \{v_1,v_6\}};
(78,0 )*+{\bullet}="t8"+(0,-3)*{\scriptstyle \{v_2,v_5\}};
(92,30)*+{\bullet}="t9"+(0,3)*{\scriptstyle \{v_2,v_6\}};
(92,15)*+{\bullet}="t10"+(-6,0)*{\scriptstyle \{v_0,v_3\}};
(92,0 )*+{\bullet}="t11"+(0,-3)*{\scriptstyle \{v_1,v_4\}};
(106,30)*+{\bullet}="t12"+(0,3)*{\scriptstyle \{v_2,v_3\}};
(106,15)*+{\bullet}="t13"+(0,3)*{\scriptstyle \{v_3,v_6\}};
(106,0 )*+{\bullet}="t14"+(0,-3)*{\scriptstyle \{v_4,v_5\}};
(120,30)*+{\bullet}="t15"+(0,3)*{\scriptstyle \{v_0,v_6\}};
(120,15)*+{\bullet}="t16"+(0,-3)*{\scriptstyle \{v_1,v_5\}};
(120,0 )*+{\bullet}="t17"+(0,-3)*{\scriptstyle \{v_5,v_6\}};
(136,30)*+{\bullet}="t18"+(0,3)*{\scriptstyle \{v_1,v_3\}};
(136,15)*+{\bullet}="t19"+(6,0)*{\scriptstyle \{v_2,v_4\}};
(136,0 )*+{\bullet}="t20"+(0,-3)*{\scriptstyle \{v_3,v_5\}};
{\ar_{1} "t0";"t1"};
{\ar_{2} "t4";"t5"};
{\ar_{1} "t7";"t8"};
{\ar_{2} "t9";"t10"};
{\ar_{1} "t10";"t11"};
{\ar_{1} "t13";"t14"};
{\ar_{1} "t15";"t16"};
{\ar_{1} "t18";"t19"};
{\ar_{2} "t19";"t20"};
\endxy
\end{center}
\caption{The graph $\U$ (on the right) corresponding to the universal graph $\Gtilde_w$ (on the left) for $w=g_1^2 g_2 g_1 g_2 g_1^{\;-1} g_2$. ($\Gtilde_w$'s vertices are denoted here by $v_1,\ldots,v_7$ while the $s_i$ labels are omitted.) }  \label{fig:Upsilon}
\end{figure}

Let $\{v_i,v_j\}$ and $\{v_k,v_l\}$ be two vertices of $\U$ (so $v_i,v_j,v_k,v_l$ are vertices of $\Gtilde_w$). Then clearly, if $\{v_i,v_j\}$ and $\{v_k,v_l\}$ belong to the same connected component of $\U$, the two sets of pairs $\{\{v_i,v_j\}\}$ and $\{\{v_k,v_l\}\}$ generate \emph{the same} quotient graph in $\Q_{w,2,A}$. Thus, if we denote by \emph{$\Co(\U)$} the set of connected components of $\U$, we can define a function
\[
f~:~\Co(\U) \to \Q_{w,2,A}
\]
that sends the component $C$ to the quotient graph in $\Q_{w,2,A}$ generated from $\Gtilde_w$ by each of the vertices of $C$. The discussion above about $\Q_{w,2,A}$ shows that $f$ is surjective. \\

Now assume there is a simple cycle in $\U$ containing the vertex $\{v_1,v_2\}$. This implies that there are identical paths in $\Gtilde_w$ from $v_1$ to $v_2$ and from $v_2$ to $v_1$ (or from $v_1$ to itself and from $v_2$ to itself). This, in turn, implies there is some periodicity in $\Gtilde_w$ which is impossible as $w$ is not a power of another word. Lemma 14 in \cite{LP10} contains a detailed proof to that when $\f(w)\geq2$, $\U$ indeed contains no cycles.

Thus, $\U$ consists of exactly $\binom{v_{\Gtilde_w}}{2} - \sum_{j=1}^k{\binom{e_{\Gtilde_w}^j}{2}}$ connected components, and since $f$ is surjective, we obtain
\[ \left|\Q_{w,2,A}\right| \le |\Co(\U)| = \binom{v_{\Gtilde_w}}{2} - \sum_{j=1}^k{\binom{e_{\Gtilde_w}^j}{2}} \] \\

\subsection{The Injectivity of the Function \texorpdfstring{$f$}{f}}
\label{sbs:injectivity}

The partial proof up to this point appeared in \cite{LP10}. The missing ingredient that we complete in the rest of Section \ref{sec:proof}, is that $f$ is actually also injective. Put differently, we need to show that for any quotient graph $\G\in\Q_{w,2,A}$, there is a unique connected component $C$ of $\U$ that generates it. We do this by constructing the inverse function, $f^{-1}$.

We say that a connected component $C\in\Co(\U)$ is \emph{realized} by a quotient graph $\G\in\Q_{w,2,A}$ if for a vertex $\{v_i,v_j\}$ of $C$ we have $v_i \equiv v_j$ in $\G$. (Note that if one vertex of $C$ has this property, then so do all vertices of $C$.)

Constructing $f^{-1}$ could have been easy if for every
$\G\in\Q_{w,2,A}$ there was a single component realized by it.
However, this is not the case. Observe, for instance, $C$, the right-most component of $\U$ in Figure \ref{fig:Upsilon}. In the quotient graph $f(C)$ we have $v_1 \equiv v_3$ and also $v_3 \equiv v_5$. By transitivity, we also have $v_1 \equiv v_5$. In this case, therefore, the component $C$ ``implies'' the component containing $\{v_1,v_5\}$, and both of them are realized by the quotient graph $f(C)$ (in particular, we also have $v_0 \equiv v_6$ in $f(C)$). In general, we define:
\begin{definition} \label{def:implications}
For $C,C`\in\Co(\U)$, we say that the component $C$ \textbf{implies} the component $C`$ (or that there is an \textbf{implication} from $C$ to $C`$) whenever $C`$ is realized by the quotient graph $f(C)$.
\end{definition}
This kind of implications (all due to transitivity of the relation $\equiv$) between different components make the challenge of constructing $f^{-1}$ more delicate.

Thus we need some more machinery in order to construct the inverse function of $f$. We first show (Section \ref{sbsbs:repetitions}) the connection between components of $\U$ and cyclic repetitions in $w$, and then (Section \ref{sbsbs:factorization}) develop a recursive factoring of words which enables us to describe the quotient graph $f(C)$. With these two, we will be able to construct the inverse function in Section \ref{sbsbs:inverse_f}.

\subsubsection{Cyclic Repetitions in \texorpdfstring{$w$}{w}} \label{sbsbs:repetitions}
There is a clear connection between components of $\U$ and cyclic repetitions in $w$. Since the degree of every vertex in $\Gtilde_w$ is two, the degree of every vertex in $\U$ is $\leq 2$. So every connected component of $\U$ is either an isolated vertex or a simple path (recall that $\U$ contains no cycles). Each such path corresponds to some maximal cyclic repetition in $w$. For instance, the middle component of $\U$ in Figure \ref{fig:Upsilon} consists of a path from $\{v_2,v_6\}$ to $\{v_1,v_4\}$. This path corresponds to a cyclic repetition of the subword $g_2g_1$ in $w=g_1^2 g_2 g_1 g_2 g_1^{\;-1} g_2$. This subword appears in the third and fourth letters, as well as in the last and first letters. It is maximal because the preceding letters (the second and sixth), as well as the following ones (the fifth and second), are different from each other. We distinguish between two types of repetitions:
\begin{definition} \label{def:coherent}
Two different appearances of a word $u$ in the word $w$, cyclically, are called a \textbf{cyclic repetition}. A cyclic repetition is called \textbf{non-coherent} if $u$ appears (cyclically) in $w$ once as $u$ and once as $u^{-1}$. Otherwise, the repetition is called \textbf{coherent}. We also describe components of $\U$ by their corresponding repetitions: isolated, coherent or non-coherent.
\end{definition}

For example, the component $C_1$ containing $\{v_2,v_6\}$ in $\U$ in Figure \ref{fig:Upsilon} is coherent as it corresponds to a maximal coherent repetition (of length $2$), whereas the component $C_2$ of $\{v_0,v_6\}$ is non-coherent: it corresponds to a maximal non-coherent repetition (of length $1$).\\

In what follows the notion of a ``repetition'' will be used to denote a maximal cyclic repetition, unless otherwise stated. We say that a repetition (or the corresponding component) \emph{is overlapping} if the two appearances of the subword $u$ meet, even if only in the endpoints (e.g., if $u$ appears in $w$ from the $i$-th letter to the $j$-th letter, and then from letter $j+1$ to letter $k$, we say that the repetition is overlapping). This is equivalent to saying that the vertices of the component are \emph{not} all disjoint (as sets of pairs).

In fact, the transitivity of $\equiv$ plays a role in the generation of $f(C)$, if and only if the the vertices of $C$ are not all disjoint. Therefore, the component $C$ implies other components if and only if it is overlapping (for this simple observation we use the fact that $C$ is not a cycle). It turns out that this is never the case when the repetition is non-coherent.

\begin{lemma} \label{lem:non_coherent}
A non-coherent repetition is never overlapping. In other words, a non-coherent component $C\in\Co(\U)$ does not imply any other components.
\end{lemma}
\begin{proof}
Recall that we wrote $w = g_{i_1}^{\alpha_1} g_{i_2}^{\alpha_2} \ldots g_{i_{|w|}}^{\alpha_{|w|}} \in \S^*$ ($i_1, i_2, \ldots, i_{|w|} \in \{1, 2, \ldots, k\}$, $\alpha_1, \alpha_2, \ldots, \alpha_{|w|} \in \{-1, 1\}$). Assume that there is a non-coherent cyclic repetition in $w$ consisting of a subword $u$, such that the two appearances of $u$ overlap. Up to a cyclic shift of $w$ we can assume that $u$ is a prefix of $w$, so that $u = g_{i_1}^{\alpha_1} g_{i_2}^{\alpha_2}\ldots g_{i_{|u|}}^{\alpha_{|u|}}$, and that there is some $1\leq j \leq |u|+1$ such that $u^{-1} = g_{i_j}^{\alpha_j} g_{i_{j+1}}^{\alpha_{j+1}} \ldots g_{i_{j+|u|-1}}^{\alpha_{j+|u|-1}}$ (the summation of indices is modulo $|w|$). In other words, $g_{i_k}^{\alpha_k}=g_{i_{j+|u|-k}}^{-\alpha_{j+|u|-k}}$ for $k=1,\ldots,|u|$. Now, if $j$ and $|u|$ have the same parity, let $k=\frac{j+|u|}{2}\leq |u|$, and we get that $g_{i_k}^{\alpha_k}=g_{i_k}^{-\alpha_k}$, which is impossible. On the other hand, if $j$ and $|u|$ have different parity, let $k=\frac{j+|u|-1}{2}\leq |u|$, and then $g_{i_k}^{\alpha_k} = g_{i_{k+1}}^{-\alpha_{k+1}}$ which is impossible because $w$ is cyclically reduced.\\
\end{proof}


\subsubsection{Recursive Factorization of a Word and the Quotient Graph \texorpdfstring{$f(C)$}{f(C)}} \label{sbsbs:factorization}
From Lemma \ref{lem:non_coherent}, we deduce that a component of $\U$ implies other components if and only if it corresponds to a coherent overlapping repetition. We continue by investigating the quotient graphs generated by any $C\in\Co(\U)$, including those components that imply others. The main tool we use is that of recursive factorization of words.

Denote by $v_0,v_1,\ldots,v_{|w|-1}$ the vertices of $\Gtilde_w$ ($v_0$ is the block $\{s_0,s_{|w|}\}$ and every other $v_i$ is the block $\{s_i\}$). For every $v_i,v_j$ we let $|v_i\to v_j|$ denote the length of the path in $\Gtilde_w$ from $v_i$ to $v_j$ which goes in the direction that $w$ traces $\Gtilde_w$ (so $|v_i\to v_j|$ equals $(j-i)~mod~|w| \in \{0,1,\ldots,|w|-1\}$).

Now let $C$ be some component in $\Co(\U)$.  We pick two vertices $v_x,v_y$ of $\Gtilde_w$ such that $\{v_x,v_y\}$ is a vertex of $C$, in the following way:
\begin{itemize}
  \item If $C$ is an isolated vertex, we let $\{v_x,v_y\}$ be this vertex (with arbitrary order of $v_x$ and $v_y$).
  \item If $C$ corresponds to a coherent repetition of the subword $u$, let $v_x$ and $v_y$ be the vertices of $\Gtilde_w$ where the two appearances of $u$ begin ($\{v_x,v_y\}$ is an endpoint in $C$), so that $|v_x \to v_y| \le |v_y \to v_x|$. We then have
    \[u = g_{i_{x+1}}^{\alpha_{x+1}} \ldots g_{i_{x+|u|}}^{\alpha_{x+|u|}} = g_{i_{y+1}}^{\alpha_{y+1}}\ldots g_{i_{y+|u|}}^{\alpha_{y+|u|}}
    \]
    (the summation of indices is modulo $|w|$).
  \item Finally, if $C$ corresponds to a non-coherent repetition of the word $u$, Lemma \ref{lem:non_coherent} shows that the two appearances of $u$ do not overlap, so we let $v_x$ denote the end of one of them (the end with respect to the direction of $w$), and $v_y$ be the beginning of the other. For instance, if $C$ is the component of $\{v_0,v_6\}$ in Figure \ref{fig:Upsilon}, we have either $v_x=v_6,~v_y=v_0$, or $v_x=v_1,~v_y=v_5$ (but not, e.g., $v_x=v_0,v_y=v_6$).
\end{itemize}

Let $a_0$ denote the subword $g_{i_{x+1}}^{\alpha_{x+1}}\ldots g_{i_y}^{\alpha_y}$, and let $b_0$ be the (cyclic) remaining of $w$, i.e. $b_0 := g_{i_{y+1}}^{\alpha_{y+1}} \ldots g_{i_x}^{\alpha_x}$. If we let $w_x$ denote the cyclic shift of $w$ by $x$ steps to the left, we can write $w_x$ as an expression in the subwords $a_0,b_0$, which we denote $w^{(0)}$, namely:
\[
w^{(0)} := a_0 b_0 = w_x
\]

If $C$ is isolated or coherent but not overlapping, the
factorization of $w$ stops at $w^{(0)}$. Note that because $u$
marked a maximal repetition, $a_0$ is not a prefix of $b_0$ nor vice
versa in this case.

The factorization of $w$ once again stops at $w^{(0)}$ when $C$ is non-coherent. Note that in this case the two appearances of the non-coherent repetition form the beginning and end of $b_0$. In particular, the first letter of $b_0$ equals the inverse of its last letter, so it cannot equal the first letter of $a_0$ (as $w$ is cyclically reduced).

The factorization process continues only when $C$ is (coherent and) overlapping. Recall that in this case, $v_x$ and $v_y$ where chosen so that $|a_0| = |v_x \to v_y| \leq |v_y \to v_x| = |b_0|$. But $a_0$ and $b_0$ cannot be of equal length, as this would yield they are both prefixes of $u$ and thus equal, contradicting $w$'s not being a power. So $|a_0|<|b_0|$ and $a_0$ is actually a prefix of $b_0$. Then, denote by $b_1$ the remaining suffix of $b_0$: $b_1 := g_{i_{y+|a_0|+1}}^{\alpha_{y+|a_0|+1}} \ldots g_{i_x}^{\alpha_x}$. We also define $a_1$ to equal $a_0$. We get that we can also write $w_x$ as the expression $w^{(1)}$ in the subwords $a_1, b_1$:
\[
w^{(1)} :=  a_1^{\;2} b_1 = w_x
\]
This factorization is depicted in Figure \ref{fig:w1}.

\begin{figure}[h]
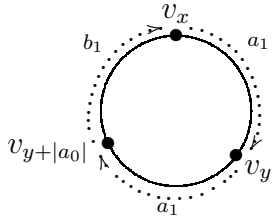

\centering
\xy
(10,10)*\xycircle(10,10){};
(10,20   )*+{\bullet}="x"+(0,3)*{v_x};;
(18, 4   )*+{\bullet}="y"+(3,-2)*{v_y};;
( 1, 5.65)*+{\bullet}="z"+(-8,-1)*{v_{y+|a_0|}};;
(12,21);(20,5) **\crv{~*=<4pt>{.} (20,20)&(23,10) } ?<(1)*\dir{>};
(21,19)*{\scriptstyle a_1};
(18,2);(0,4) **\crv{~*=<4pt>{.} (10,-6)&(1,1) } ?<(1)*\dir{>};
(9,-3)*{\scriptstyle a_1};
(-1,6);(8,21) **\crv{~*=<4pt>{.} (-3,10)&(0,20) } ?<(1)*\dir{>};
(-1,19)*{\scriptstyle b_1};
\endxy
\caption{A sketch of the factorization $w^{(1)}:~ w_x=a_1^{\;2}b_1$ as seen in the graph $\Gtilde_w$.}\label{fig:w1}
\end{figure}

From $w^{(1)}$ we continue recursively to obtain more refined
factorizations of $w_x$, until we get a factorization which enables
us to accurately describe the quotient graph $f(C)$. This recursive
process generalizes the step we made to obtain $w^{(1)}$ from
$w^{(0)}$ (including the very decision of whether we continue at all
the factorization process or stop at $w^{(0)}$). Each factorization
$w^{(n)}$ will consist of an expression in two non-identical
subwords $a_n$ and $b_n$. Now assume we have constructed the
factorization $w^{(n)}$ ($n\geq0$). We continue by the following
scenarios:
\begin{itemize}
  \item If $a_n$ is not a prefix of $b_n$ and vice versa, stop.
  \item If $a_n$ is a prefix of $b_n$, construct $w^{(n+1)}$ by defining $a_{n+1}:=a_n,~b_{n+1}:=a_n^{\;-1}b_n$ (reduced, so $|b_{n+1}| = |b_n| - |a_n|$). Then the expression of $w^{(n+1)}$ is obtained from $w^{(n)}$ by replacing each $a_n$ with $a_{n+1}$ and each $b_n$ with $a_{n+1}b_{n+1}$.
  \item If $b_n$ is a prefix of $a_n$, construct $w^{(n+1)}$ by defining $a_{n+1}:=b_n,~b_{n+1}:=b_n^{\;-1}a_n$ (reduced, so $|b_{n+1}|=|a_n|-|b_n|$). Then the expression of $w^{(n+1)}$ is obtained from $w^{(n)}$ by replacing each $a_n$ with $a_{n+1}b_{n+1}$ and each $b_n$ with $a_{n+1}$.
\end{itemize}

For example, if $a_1$ is a prefix of $b_1$, we have
$w^{(2)}=a_2^{\;3}b_2$. If $b_1$ is a prefix of $a_1$, we have
$w^{(2)}=a_2b_2a_2b_2a_2$. Note that this process always ends
because $|a_n|+|b_n|$ keeps decreasing in every step. Denote by $N$
the ordinal of the last step (so $w^{(N)}$ is the expression for
which the process halts). In particular, $N=0$ whenever $C$ does not
imply any other component, i.e. $C$ does not correspond to an
overlapping (coherent) repetition.

For each $n=0,1,\ldots,N$ we let $\G^{(n)}$ denote the (not necessarily realizable) quotient graph of $\Gtilde_w$ which has the topological structure of Figure-Eight, with one loop corresponding to $a_n$ and the other corresponding to $b_n$. This is illustrated in Figure \ref{fig:gamma_n}.

\begin{figure}[h]
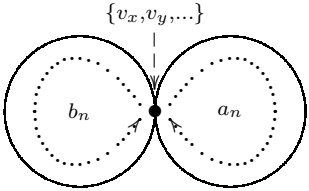

\centering
\xy
(10,10)*\xycircle(10,10){};
(30,10)*\xycircle(10,10){};
(20,10)*+{\bullet}="v1";
(20,23)*+{\scriptstyle \{v_x,v_y,\ldots\}}="t1";
{\ar@{-->} "t1";(20,13)};
(18,11);(18,9) **\crv{~*=<4pt>{.} (10,22)&(-2,10)&(10,-2)} ?<(1)*\dir{>};
(10,10)*{\scriptstyle b_n};
(22,11);(22,9) **\crv{~*=<4pt>{.} (30,22)&(42,10)&(30,-2)} ?<(1)*\dir{>};
(30,10)*{\scriptstyle a_n};
\endxy
\caption{A sketch of the (not necessarily realizable) quotient graph
$\G^{(n)}$. The vertices of this graph can be thought of as a
partition of the vertices of $\Gtilde_w$, and then the block
corresponding to the drawn vertex contains, among others, $v_x$ and
$v_y$ (when $n=0$ this block is exactly
$\{v_x,v_y\}$).}\label{fig:gamma_n}
\end{figure}

The graphs $\G^{(n)}$ can be thought of as phases in the process of generating $f(C)$, the realizable quotient graph of $\Gtilde_w$, from the set of pairs $\{\{v_x,v_y\}\}$. (This process is equivalent to the one describe in Figure \ref{fig:generating_set_of_pairs}, although we now confine ourselves to a specific order of merging). The following lemma states this in a more formal manner. In the lemma, we use the notations $h(v)$ and $t(v)$ to describe the first and last letters, respectively, of a word $v$ (with the right power $\pm1$, e.g. $h(w)=g_{i_1}^{\alpha_1}~$).

\begin{lemma} \label{lem:properties}
For every $n=0,1,\ldots,N$, the expression $w^{(n)}$, the corresponding subwords $a_n,b_n$ and the graph $\G^{(n)}$ satisfy the following list of properties:
\begin{enumerate}
  \item $w^{(n)}$ expresses $w_x$ in $\S^*$ (i.e. the equality $w^{(n)}=w_x$ holds without reductions).
  \item $a_n \ne b_n$
  \item $b_n$ does not appear twice in a row, cyclically, in $w^{(n)}$.
  \item if $n\ge1$, $a_n$ does appear twice in a row, cyclically, in $w^{(n)}$; $a_0$ appears exactly once in $w^{(0)}$.
  \item $t(a_n)^{-1}\ne h(b_n),~t(b_n)^{-1}\ne h(a_n),~t(a_n)^{-1}\ne h(a_n),~t(a_n)\neq t(b_n)$
  \item The graph $f(C)$ is a quotient of the graph $\G^{(n)}$, i.e. $f(C)$ represents a partition of the vertices of $\Gtilde_w$ which is coarser than the one of $\G^{(n)}$. Put differently, for every two vertices $v_i,v_j$ of
  $\Gtilde_w$
  \[ v_i\stackrel[\G^{(n)}]{}{\equiv} v_j ~~~ \Longrightarrow ~~~ v_i\stackrel[f(C)]{}{\equiv} v_j \]
\end{enumerate}
\end{lemma}
\begin{proof}
Property $(2)$ holds lest we obtain that $w$ is a power of $a_n$,
contradicting the assumption that $\f(w)\geq2$. Property $(1)$ holds
for $n=0$ and the recursive process clearly makes it hold for every
$n$. Properties $(3)$ and $(4)$ are valid for $n=0,1$, and can be
shown to hold for every $n=2,\ldots,N$ by simple induction.
$t(a_n)^{-1}\ne h(b_n)$ and $t(b_n)^{-1}\ne h(a_n)$ because $a_n
b_n$ and $b_n a_n$ appear as cyclic subwords of the cyclically
reduced word $w$.

When $C$ is non-coherent, the choice of $v_x$ and $v_y$ guarantees that $t(a_0)^{-1} \neq h(a_0)$. When $C$ is an isolated vertex, $t(a_0)^{-1} \neq h(a_0)$ lest $C$ could not be isolated. If $C$ is coherent, then $t(a_0)^{-1}=g_{i_y}^{-\alpha_y}\ne g_{i_{y+1}}^{\alpha_{y+1}} = h(a_0)$. When $n\geq1$, $a_n a_n$ appears as cyclic subwords of the cyclically reduced word $w$, and so $t(a_n)^{-1}\ne h(a_n)$.

When $C$ is non-coherent, $b_0$ is reduced but not cyclically reduced and thus $t(b_0) = h(b_0)^{-1} \neq t(a_0)$. When $C$ is isolated or coherent, $t(a_0) \neq t(b_0)$ because $v_x$ and $v_y$ mark the \emph{beginning} endpoints of the maximal repetition corresponding to $C$. That this property continues to hold for every $n=1,\ldots,N$ is clear because as sets $\{t(a_{n+1}),t(b_{n+1})\}=\{t(a_n),t(b_n)\}$, and we obtain $(5)$.

Property $(6)$ obviously holds for $n=0$ (the only pair of vertices of $\Gtilde_w$ which are equivalent in $\G^{(0)}$ are $v_x$ and $v_y$). Now assume it is true for some $n$. If $b_n$ is a prefix of $a_n$, $\G^{(n+1)}$ is obtained from $\G^{(n)}$ by $|b_n|$ necessary merges: we first merge the first edge and vertex of the loop representing $a_n$ with the first edge and vertex of the loop representing $b_n$, because both edges, while sharing a common end-point, are labeled and directed according to $h(b_n)=h(a_n)$. We then merge the second edges (and vertices) of both loops for the same reason, etc. After $|b_n|$ merges of this kind, we obtain exactly $\G^{(n+1)}$. Obviously, the case where $a_n$ is a prefix of $b_n$ is completely equivalent. This completes the proof of property $(6)$.
\end{proof}

Recall that we aim to describe the quotient graph $f(C)$. $\G^{(N)}$ is a good approximate, but some few extra merges might be needed to obtain a realizable quotient graph. We do know that $a_N$ and $b_N$ are cyclic subwords of $w$, so they are both reduced. Thus, the only vertex in $\G^{(N)}$ which might have two equally labeled edges exiting it (or entering it) is the one of degree $4$ which marks the endpoints of $a_N$ and $b_N$. Four edges are incident with this vertex: $h(a_N), t(a_N)^{-1}, h(b_N), t(b_N)^{-1}$. Property $(5)$ of Lemma \ref{lem:properties} shows that there are three possibilities for equality relations among these four edges: either they are all different, or the only equality is $h(a_N) = h(b_N)$, or the only equality is $h(b_N) = t(b_N)^{-1}$. (These two equalities cannot coexist because this would yield that $h(a_N) = t(b_N)^{-1}$, contradicting Property $(5)$).

In the first case, $\G^{(N)}$ is realizable, and by Property $(6)$ we obtain $f(C) = \G^{(N)}$. This is a Figure-Eight graph. In the second case, $a_N$ and $b_N$ have some common prefix $c_N$. But $a_N$ and $b_N$ are not prefixes of each other, so this common prefix is strictly shorter than both of them. $f(C)$ has therefore Theta-shape. Finally, in the last case, $b_N$ is reduced but not cyclically reduced, so we can write $b_N$ as $d_N e_N d_N^{-1}$, with $e_N$ cyclically reduced. The shape of $f(C)$ in this case is that of a Barbell. Those three cases are illustrated in Figure \ref{fig:f(C)}.

\begin{figure}[h]
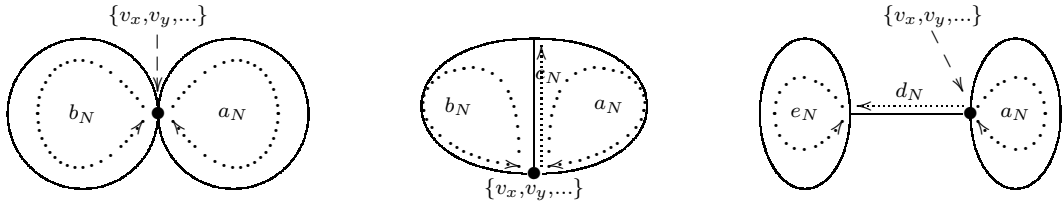

\centering
\xy
(10,10)*\xycircle(10,10){};
(30,10)*\xycircle(10,10){};
(20,10)*+{\bullet}="v1";
(20,23)*+{\scriptstyle \{v_x,v_y,\ldots\}}="t1";
{\ar@{-->} "t1";(20,13)};
(18,11);(18,9) **\crv{~*=<4pt>{.} (10,22)&(-2,10)&(10,-2)} ?<(1)*\dir{>};
(10,10)*{\scriptstyle b_N};
(22,11);(22,9) **\crv{~*=<4pt>{.} (30,22)&(42,10)&(30,-2)} ?<(1)*\dir{>};
(30,10)*{\scriptstyle a_N};
(70,2)*{\bullet}="v2"+(0,-2)*{\scriptstyle \{v_x,v_y,\ldots\}};
(70,20)*{}="t2";
"v2";"t2" **\dir{-};
"t2";"v2" **\crv{(60,20)&(50,11)&(60,2)};
"t2";"v2" **\crv{(80,20)&(90,11)&(80,2)};
{\ar@{.>} (71,3);(71,19)};
(72,15)*{\scriptstyle c_N};
(68,5);(68,3) **\crv{~*=<4pt>{.} (68,15)&(60,18)&(52,11)&(60,4)} ?<(1)*\dir{>};
(60,11)*{\scriptstyle b_N};
(72,5);(72,3) **\crv{~*=<4pt>{.} (72,15)&(80,18)&(88,11)&(80,4)} ?<(1)*\dir{>};
(80,11)*{\scriptstyle a_N};
(128,10)*{\bullet}="v3"+(-5,13)*{\scriptstyle \{v_x,v_y,\ldots\}};
{\ar@{-->} (123,21);(127,13)};
(112,10)*{}="t3";
"v3";"t3" **\dir{-};
(106,10)*\xycircle(6,10){};
(134,10)*\xycircle(6,10){};
(129,11);(129,9) **\crv{~*=<4pt>{.} (134,18)&(142,10)&(134,2)} ?<(1)*\dir{>};
(134,10)*{\scriptstyle a_N};
{\ar@{.>} (127,11);(113,11)};
(120,13)*{\scriptstyle d_N};
(111,11);(111,9) **\crv{~*=<4pt>{.} (106,18)&(98,10)&(106,2)} ?<(1)*\dir{>};
(106,10)*{\scriptstyle e_N};
\endxy
\caption{The three possible shapes of the quotient graph
$f(C)\in\Q_{w,2,A}$, when $C\in\Co(\U)$. The left graph is
Figure-Eight and it corresponds to the case where $f(C)$ is exactly
the graph $\G^{(N)}$. The middle graph is Theta-shaped, and it
depicts the case where $a_N$ and $b_N$ have a common prefix $c_N$.
The Barbell graph on the right corresponds to the case where $b_N$
is reduced but not cyclically reduced, and can be expressed as
$d_Ne_Nd_N^{-1}$ with $e_N$ cyclically reduced.}\label{fig:f(C)}
\end{figure}

\subsubsection{The Inverse Function of \texorpdfstring{$f$}{f}} \label{sbsbs:inverse_f}

We now have all the machinery necessary for suggesting a definition
for the inverse function of $f$: \[h~:~\Q_{w,2,A}\to \Co(\U)\] Let
$\G\in\Q_{w,2,A}$ be some quotient graph. Recall that $f$ is
surjective, so $\G = f(C)$ for some $C \in \Co(\U)$. The definition
of $h$ is based on a simple algorithm that recovers, in a
deterministic fashion and without preknowledge of $C$, the only
possible chain of expressions $w^{(N)}, \ldots, w^{(0)}$ (in this
reversed order, and up to some cyclic shifts), and eventually also
the only possible component $C$ that yielded $\G$. This will show
that $f$ indeed has an inverse function. \\

By the discussion in Section \ref{sbsbs:factorization}, $\G$ is of
one of the three shapes in Figure \ref{fig:f(C)}. Since we can trace
the path of $w$ along $\G$, we can determine by Properties $(3)$ and
$(4)$ of Lemma \ref{lem:properties} if we are in the case $N=0$ or
not.

Assume first that we are in the case $N>0$. Then, by Properties
$(3)$ and $(4)$ of \ref{lem:properties} we can recover $a_N$ and
$b_N$ and also distinguish between them. We can also locate the
vertex of $\G$ corresponding to $\{v_x,v_y,\ldots\}$ (see Figure
\ref{fig:f(C)}). We can thus recover the graph $\G^{(N)}$, and up to
a cyclic shift also $w^{(N)}$.

Note that in the recursive process in which we construct $w^{(N)}$,
whenever $a_n$ is a prefix of $b_n$ ($n\geq1$), we get that
$a_{n+1}$ appears (at least) thrice in a row, cyclically, in
$w^{(n+1)}$, while if $b_n$ is a prefix of $a_n$, then $a_{n+1}$
appears twice in a row but not thrice, cyclically, in $w^{(n+1)}$.
When $N\geq1$, $a_0$ is a prefix of $b_0$. Therefore, given
$w^{(n+1)}$ (even up to a cyclic shift), for $n\geq0$, we can always
determine which of the two kinds of steps in the recursive process
was used to obtain $w^{(n+1)}$ from $w^{(n)}$. Since these steps are
easily reversible, we can recover $w^{(n)}$ from $w^{(n+1)}$ (up to
the same cyclic shift). (E.g., if the second kind of recursive step
was used to obtain $w^{(n+1)}$ from $w^{(n)}$, then we can recover
$a_n$ and $b_n$ from $a_{n+1},b_{n+1}$ by defining $b_n:=a_{n+1},
~a_n:=a_{n+1}b_{n+1}$. $w^{(n)}$ can be recovered from $w^{(n+1)}$
by replacing each appearance of the sequence $a_{n+1}b_{n+1}$ with
$a_n$, and every other appearance of $a_{n+1}$ with $b_n$.)

Thus, when $N>0$ we can recover $w^{(0)}$ up to some cyclic shift.
The beginning points of $a_0$ and $b_0$ in $\Gtilde_w$ are the
vertices $v_x$ and $v_y$. We let $h(\G)$ be the component
$C\in\Co(\U)$ that contains the vertex $\{v_x,v_y\}$.

Finally, when $N=0$, it is easy to observe that all the pairs of
vertices of $\Gtilde_w$ that are equivalent in $\G$, all belong to
the same component of $\U$. We let $h(\G)$ be this component.

For every $C\in\Co(\U)$, we see therefore that $h$ manages to
recover $C$ from $f(C)$, so $h(f(C))=C$. Since $f$ was shown to be
surjective, it follows that $h$ is the inverse of $f$. Thus
$|Q_{2,w,A}|=|\Co(\U)|$, and we obtain the desired equality in
\eqref{eq:size_of_Qw2a}. This proves Lemma \ref{lem:balancing}, and
thus completes the proof of Theorem \ref{ther:phi=beta=2} and of
Corollary \ref{cor:E>1-1/n^2}.

\bibliographystyle{amsalpha}
\bibliography{PhiBeta}{}

\end{document}